\documentclass[12pt]{amsart}

\usepackage[left=.9in, right=.9in, top=.9in, bottom=.9in]
{geometry}
\usepackage{amsfonts,amssymb,graphicx,amsmath,amsthm, float}
\usepackage{xypic}
\usepackage[all]{xy}
\theoremstyle{plain}
\newtheorem{theorem}{Theorem}[section]

\newtheorem{lemma}[theorem]{Lemma}
\newtheorem{example}[theorem]{Example}

\newtheorem{corollary}[theorem]{Corollary}

\newtheorem{proposition}[theorem]{Proposition}

\newcommand{\R}{\mathbb{R}}
\newcommand{\N}{\mathbb{N}}

\theoremstyle{remark}
\newtheorem*{remark}{Remark}

\newcommand{\reg}{\text{Reg}(R)}
\newcommand{\treg}{\tau_{\text{reg}}}

\begin{document}

\title{$\tau$-Regular Factorization in Commutative Rings with Zero-Divisors}
        \date{\today}

\author{Christopher Park Mooney}
\address{R016 Reinhart Center, Dept. of Mathematics, Viterbo University, La Crosse, WI 54601}
\email{cpmooney@viterbo.edu}


\begin{abstract}
Recently there has been a flurry of research on generalized factorization techniques in both integral domains and rings with zero-divisors, namely $\tau$-factorization.  There are several ways that authors have studied factorization in rings with zero-divisors.  This paper focuses on the method of regular factorizations introduced by D.D. Anderson and S. Valdes-Leon.  We investigate how one can extend the notion of $\tau$-factorization to commutative rings with zero-divisors by using the regular factorization approach.  The study of regular factorization is particularly effective because the distinct notions of associate and irreducible coincide for regular elements.  We also note that the popular U-factorization developed by C.R. Fletcher also coincides since every regular divisor is essential.  This will greatly simplify many of the cumbersome finite factorization definitions that exist in the literature when studying factorization in rings with zero-divisors.  
\\
\vspace{.1in}\noindent \textbf{2010 AMS Subject Classification:} 13A05, 13E99, 13F15
\end{abstract}
\maketitle
\section{Introduction}  
\indent \indent There has been a considerable amount of research done on the factorization properties of commutative rings, especially domains.  Unique factorization domains (UFDs) are well understood and have been studied extensively over the years.  More recently, many authors have studied rings which satisfy various weakenings of the UFD conditions.  These factorization properties of domains have been extended in several distinct ways to rings with zero-divisors.  Traditionally, in the domain case, authors have studied prime or irreducible factorizations. More recently, research has been done on generalizing the types of factorizations that have been studied to include things like co-maximal factorizations or using $\star$-operations to generalize factorization.  
\\
\indent Of particular interest to the current article is the 2011 work of D.D. Anderson and A. Frazier.  This is a survey article, \cite{Frazier}, on the study of factorization in domains in which the authors introduce $\tau$-factorization.  The use of $\tau$-factorization yields a beautiful synthesis of many of these generalizations of factorizations studied in the integral domain case.  In many ways, this article was able to consolidate all of the factorization research in integral domains into a single method of studying factorization.  Recently, the author has begun to study methods of extending this powerful approach of $\tau$-factorization to the case of a commutative ring with zero-divisors.  Because of the numerous approaches that have been taken to study factorization in rings with zero-divisors, this has led to many approaches to extending $\tau$-factorization.  
\\
\indent In \cite{Mooney}, the author used the methods established by D.D. Anderson and S. Valdes-Leon in \cite{Valdezleon} to extend many of the $\tau$-factorization definitions to work also in rings with zero-divisors.  In \cite{Mooney2}, the author investigated extending $\tau$-factorization using the notion of U-factorizations developed first by C.R. Fletcher in \cite{Fletcher, Fletcher2} and then studied extensively by M. Axtell, N. Baeth, and J. Stickles in \cite{Axtell, Axtell2}.   In \cite{Mooneycomplete}, the author studied yet another approach to extending $\tau$-factorization, by using complete factorizations which was touched on in \cite{Frazier} in the case of integral domains.  
\\
\indent In the present article, we concentrate on the approach studied in \cite[Section 5]{Valdezleon2} in which D.D. Anderson and S. Valdes-Leon study what was called regular factorization.  This approach takes advantage of the fact that for regular elements, all of the traditionally distinct associate relations and irreducible elements behave as they do in integral domains, where they all are equivalent once again.  We see that this approach will greatly simplify matters and in fact unifies many of the previous methods in \cite{Mooney, Mooney2}.
\\
\indent  In Section \ref{sec: prelim}, we provide some necessary background definitions and theorems.  In Section \ref{sec: regular}, we develop many of the definitions of $\tau$-regular-factorization, $\tau$-regular irreducible elements as well as $\tau$-regular finite factorization properties that rings may have.  This is done by using the approach of D.D. Anderson and S. Valdez-Leon in \cite[Section 5]{Valdezleon2}, where they restrict their study of $\tau$-factorization to only the regular elements of a commutative ring with $1$.  In Section \ref{sec: tau-regular-results}, we prove several theorems which describe the relationships between the various $\tau$-regular finite factorization properties that rings may possess.  In Section \ref{sec: tau-regular}, we compare this new method of extending $\tau$-factorization with the previous work in \cite{Mooney} and the relation $\tau_r:=\tau \cap \reg \times \reg$.  In Section \ref{sec: other props}, we demonstrate how these $\tau$-regular finite factorization properties are related to other finite factorization properties defined in other works, especially \cite{Mooney} and \cite{Mooney2}. 

\section{Preliminary Definitions and Results} \label{sec: prelim}
\indent We will assume $R$ is a commutative ring with $1 \neq 0$.   Let $R^*=R-\{0\}$, let $U(R)$ be the set of units of $R$, and let $R^{\#}=R^*-U(R)$ be the non-zero, non-units of $R$.  As in \cite{Valdezleon}, we let $a \sim b$ if $(a)=(b)$, $a\approx b$ if there exists $\lambda \in U(R)$ such that $a=\lambda b$, and $a\cong b$ if (1) $a\sim b$ and (2) $a=b=0$ or if $a=rb$ for some $r\in R$ then $r\in U(R)$.  We say $a$ and $b$ are \emph{associates} (resp. \emph{strong associates, very strong associates}) if $a\sim b$ (resp. $a\approx b$, $a \cong b$).  As in \cite{Stickles}, a ring $R$ is said to be \emph{strongly associate} (resp. \emph{very strongly associate}) ring if for any $a,b \in R$, $a\sim b$ implies $a \approx b$ (resp. $a \cong b$).

\subsection{$\tau$-Factorization in Rings with Zero-Divisors}\ 
\\
\indent Let $\tau$ be a relation on $R^{\#}$, that is, $\tau \subseteq R^{\#} \times R^{\#}$.  We will always assume further that $\tau$ is symmetric.  For non-units $a, a_i \in R$, and $\lambda \in U(R)$, $a=\lambda a_1 \cdots a_n$ is said to be a \emph{$\tau$-factorization} if $a_i \tau a_i$ for all $i\neq j$.  If $n=1$, then this is said to be a trivial \emph{$\tau$-factorization}.  
\\
\indent As in \cite{Mooney}, we say $\tau$ is \emph{multiplicative} (resp. \emph{divisive}) if for $a,b,c \in R^{\#}$ (resp. $a,b,b' \in R^{\#}$), $a\tau b$ and $a\tau c$ imply $a\tau bc$ (resp. $a\tau b$ and $b'\mid b$ imply $a \tau b'$).  We say $\tau$ is \emph{associate} (resp. \emph{strongly associate}, \emph{very strongly associate) preserving} if for $a,b,b'\in R^{\#}$ with $b\sim b'$ (resp. $b\approx b'$, $b\cong b'$) $a\tau b$ implies $a\tau b'$.  A \emph{$\tau$-refinement} of a $\tau$-factorization $\lambda a_1 \cdots a_n$ is a $\tau$-factorization of the form
$$(\lambda \lambda_1 \cdots \lambda_n)b_{11}\cdots b_{1m_1}\cdot b_{21}\cdots b_{2m_2} \cdots b_{n1} \cdots b_{nm_n}$$
where $a_i=\lambda_ib_{i_1}\cdots b_{i_{m_i}}$ is a $\tau$-factorization for each $i$.  We say that $\tau$ is \emph{refinable} if every $\tau$-refinement of a $\tau$-factorization is a $\tau$-factorization.  We say $\tau$ is \emph{combinable} if whenever $\lambda a_1 \cdots a_n$ is a $\tau$-factorization, then so is each $\lambda a_1 \cdots a_{i-1}(a_ia_{i+1})a_{i+2}\cdots a_n$.  
\\
\indent We now pause to supply the reader with a few examples of particularly useful or interesting $\tau$-relations to give an idea of the power of $\tau$-factorization.

\begin{example} Let $R$ be a commutative ring with $1$.
\begin{enumerate}

\item $\tau=R^{\#}\times R^{\#}$.  This yields the usual factorizations in $R$ and $\mid_{\tau}$ is the same as the usual divides.  $\tau$ is multiplicative and divisive (hence associate preserving as we shall soon see).

\item $\tau=\emptyset$.  For every $a\in R^{\#}$, there is only the trivial factorization and $a\mid{_\tau} b \Leftrightarrow a=\lambda b$ for $\lambda \in U(R)$ $\Leftrightarrow a\approx b$.  Again $\tau$ is both multiplicative and divisive (vacuously).

\item Let $S$ be a non-empty subset of $R^{\#}$ and let $\tau=S\times S$, $a\tau b \Leftrightarrow a,b\in S$. So $\tau$ is multiplicative (resp. divisive) if and only if $S$ is multiplicatively closed (resp. closed under non-unit factors).  A non-trivial $\tau$-factorization is up to unit factors a factorization into elements from $S$.


\item Let $\star$ be a star-operation on $R$ and define $a\tau b \Leftrightarrow (a,b)^{\star}=R$, that is $a$ and $b$ are $\star$-coprime or $\star$-comaximal.  This particular operation has been studied more in depth by Jason Juett in \cite{Juettcomax}.  When $\star=d$, the identity star operation, we get the co-maximal factorizations of S. McAdam and R. Swan, in \cite{Mcadam}.


\item Let $a \tau_z b \Leftrightarrow ab=0$.  Then every $a\in R^{\#}$ is a $\tau$-atom.  The only nontrivial $\tau$-factorizations are $0=\lambda a_1 \cdot \ldots \cdot a_n$ where $a_i \cdot a_j=0$ for all $i \neq j$.  This example was studied extensively in \cite{Mooney} and has a close relationship with zero-divisor graphs.

\item Let $a\tau b \Leftrightarrow a,b\in\text{Reg}(R)$.  Then this gives us the regular factorization studied in \cite{Valdezleon3}.  This is the inspiration for Section \ref{sec: regular}.

\item Let $\tau \subseteq R^{\#}\times R^{\#}$, then we define $\tau_{reg}:=\tau \cap \left(Reg(R) \times Reg(R)\right)$.  Because the collection of regular elements is a saturated, multiplicatively closed set, this has the effect of only allowing trivial factorizations of the zero-divisors.  This is the type of $\tau$-factorization we would like to use to compare with the notion of $\tau$-factorizations by way of the regular factorizations studied in \cite{Valdezleon3}.  This will be studied more in depth in Section \ref{sec: tau-regular}.
\end{enumerate}
\end{example}
\indent We now summarize several of the definitions given in \cite{Mooney} and \cite{Mooneycomplete}.  Let $a\in R$ be a non-unit.  Then $a$ is said to be \emph{$\tau$-irreducible} or \emph{$\tau$-atomic} if for any $\tau$-factorization $a=\lambda a_1 \cdots a_n$, we have $a\sim a_i$ for some $i$.  We will say $a$ is \emph{$\tau$-strongly irreducible} or \emph{$\tau$-strongly atomic} if for any $\tau$-factorization $a=\lambda a_1 \cdots a_n$, we have $a \approx a_i$ for some $a_i$.  We will say that $a$ is \emph{$\tau$-m-irreducible} or \emph{$\tau$-m-atomic} if for any $\tau$-factorization $a=\lambda a_1 \cdots a_n$, we have $a \sim a_i$ for all $i$.  Note: the $m$ is for ``maximal" since such an $a$ is maximal among principal ideals generated by elements which occur as $\tau$-factors of $a$.  As in \cite{Mooneycomplete}, $a\in R$ is said to be a \emph{$\tau$-unrefinable atom} if $a$ admits only trivial $\tau$-factorizations.  We will say that $a$ is \emph{$\tau$-very strongly irreducible} or \emph{$\tau$-very strongly atomic} if $a\cong a$ and $a$ has no non-trivial $\tau$-factorizations.  See \cite{Mooney} and \cite{Mooneycomplete} for more equivalent definitions of these various forms of $\tau$-irreducibility.  

\indent We have the following relationship between the various types of $\tau$-irreducibles which is proved in \cite[Theorem 3.9]{Mooney} as well as \cite{Mooneycomplete}.
\begin{theorem}\label{thm: unrefinable} Let $R$ be a commutative ring with $1$ and $\tau$ be a symmetric relation on $R^{\#}$.  Let $a \in R$ be a non-unit.  The following diagram illustrates the relationship between the various types of $\tau$-irreducibles $a$ might satisfy where $\approx$ represents $R$ being a strongly associate ring.
$$
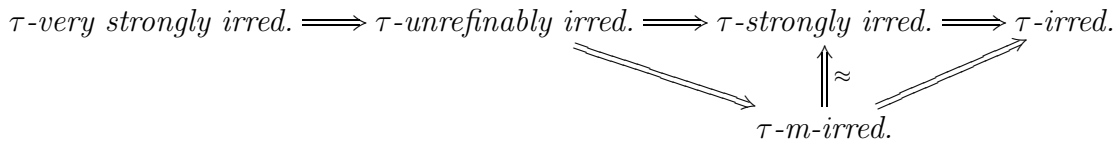
\xymatrix{
\tau\text{-very strongly irred.} \ar@{=>}[r] & \tau\text{-unrefinably irred.}\ar@{=>}[dr] \ar@{=>}[r] & \tau\text{-strongly irred.} \ar@{=>}[r]& \tau \text{-irred.}\\
& & \tau\text{-m-irred.}\ar@{=>}[u]_{\approx}\ar@{=>}[ur]   &}$$
\end{theorem} 
\indent Following A. Bouvier, a ring $R$ is said to be \emph{pr\'esimplifiable} if $x=xy$ implies $x=0$ or $y\in U(R)$ as in \cite{bouvier71, bouvier72a, bouvier72b, bouvier74}.  When $R$ is pr\'esimplifiable, the various associate relations coincide.  As seen in \cite{Mooney}, for non-zero elements, if $R$ is pr\'esimplifiable, then $\tau$-irreducible will imply $\tau$-very strongly irreducible and the various types of irreducible elements will also coincide.  Any integral domain or quasi-local ring is pr\'esimplifiable.  Examples are given in \cite{Valdezleon} and abound in the literature which show that in a general commutative ring setting, each of these types of irreducible elements are distinct.  For further discussion of the different $\tau$-irreducible elements, the reader is directed to \cite{Mooney}.
\\
\indent This leads to the following $\tau$-finite factorization properties that a commutative ring may possess given a particular choice for $\tau$, defined in \cite{Mooney, Mooneycomplete}.  Let $\alpha \in \{$atomic, strongly atomic, m-atomic, unrefinably atomic, very strongly atomic$ \}$, $\beta \in \{$associate, strong associate, very strong associate$\}$ and $\tau$ a symmetric relation on $R^{\#}$.  Then $R$ is said to be \emph{$\tau$-$\alpha$} if every non-unit $a\in R$ has a $\tau$-factorization $a=\lambda a_1\cdots a_n$ with $a_i$ being $\tau$-$\alpha$ for all $1\leq i \leq n$.  We will call such a factorization a \emph{$\tau$-$\alpha$-factorization}.  We say $R$ satisfies the \emph{$\tau$-ascending chain condition on principal ideals (ACCP)} if for every chain $(a_0) \subseteq (a_1) \subseteq \cdots \subseteq (a_i) \subseteq \cdots$ with $a_{i+1} \mid_{\tau} a_i$, there exists an $N\in \N$ such that $(a_i)=(a_N)$ for all $i>N$.
\\
\indent A ring $R$ is said to be a \emph{$\tau$-$\alpha$-$\beta$-unique factorization ring (UFR)} if (1) $R$ is $\tau$-$\alpha$ and (2) for every non-unit $a \in R$ any two $\tau$-$\alpha$ factorizations $a=\lambda_1 a_1 \cdots a_n = \lambda_2 b_1 \cdots b_m$ have $m=n$ and there is a rearrangement so that $a_i$ and $b_i$ are $\beta$.  A ring $R$ is said to be a \emph{$\tau$-$\alpha$-half factorization ring or half factorial ring (HFR)} if (1) $R$ is $\tau$-$\alpha$ and (2) for every non-unit $a \in R$ any two $\tau$-$\alpha$-factorizations have the same length.  A ring $R$ is said to be a \emph{$\tau$-bounded factorization ring (BFR)} if for every non-unit $a \in R$, there exists a natural number $N(a)$ such that for any $\tau$-factorization $a=\lambda a_1 \cdots a_n$, $n \leq N(a)$. A ring $R$ is said to be a \emph{$\tau$-$\beta$-finite factorization ring (FFR)} if for every non-unit $a \in R$ there are only a finite number of non-trivial $\tau$-factorizations up to rearrangement and $\beta$.  A ring $R$ is said to be a \emph{$\tau$-$\beta$-weak finite factorization ring (WFFR)} if for every non-unit $a \in R$, there are only finitely many $b\in R$ such that $b$ is a non-trivial $\tau$-divisor of $a$ up to $\beta$.  A ring $R$ is said to be a \emph{$\tau$-$\alpha$-$\beta$-divisor finite ring (df ring)} if for every non-unit $a \in R$, there are only finitely many $\tau$-$\alpha$ $\tau$-divisors of $a$ up to $\beta$.
\\
\indent These result in the following diagram accompanying \cite[Theorem 4.1]{Mooney} illustrating the relationship between the various $\tau$-finite factorization properties in rings with zero-divisors, where $\nabla$ represents $\tau$ being refinable.
$$
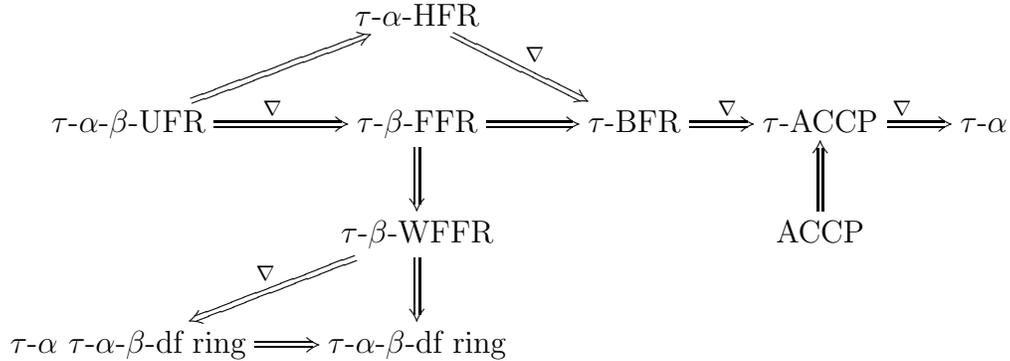
\xymatrix{
            &        \tau\text{-}\alpha \text{-HFR} \ar@{=>}^{\nabla}[dr]     &             &                  &                 \\
\tau\text{-}\alpha\text{-} \beta \text{-UFR} \ar@{=>}[ur] \ar@{=>}^{\nabla}[r]  & \tau\text{-}\beta \text{-FFR} \ar@{=>}[r] \ar@{=>}[d]  & \tau\text{-BFR} \ar@{=>}[r]^{\nabla}& \tau\text{-ACCP} \ar@{=>}^{\nabla}[r]& \tau\text{-}\alpha\\
            & \tau\text{-}\beta\text{-WFFR} \ar@{=>}[d] \ar@{=>}[dl]_{\nabla}                    &              &  \text{ACCP} \ar@{=>}[u]               &                  \\
\tau\text{-}\alpha\  \tau\text{-}\alpha\text{-}\beta \text{-df ring} \ar@{=>}[r]           & \tau\text{-}\alpha\text{-}\beta \text{-df ring} &
            }$$
\subsection{$\tau$-U-Factorization Definitions}\ 
\\
\indent In this section we briefly present the requisite $\tau$-U-factorization definitions and results from \cite{Mooney2}.  As in \cite{Axtell2}, we define U-factorization as follows.  Let $a\in R$ be a non-unit.  If $a=\lambda a_1\cdots a_n b_1\cdots b_m$ is a factorization with $\lambda \in U(R)$, $a_i,b_i \in R^{\#}$, then we will call 
$$a=\lambda a_1 a_2 \cdots a_n \left\lceil b_1 b_2\cdots b_m \right\rceil$$
 a $\emph{U-factorization}$ of $a$ if (1) $\ a_i(b_1 \cdots b_m)=(b_1 \cdots b_m)$ for all $1\leq i \leq n$ and (2) $\ b_j(b_1 \cdots \widehat{b_j} \cdots b_m) \neq (b_1 \cdots \widehat{b_j} \cdots b_m)$ for $1 \leq j \leq m$ where $\widehat{b_j}$ means $b_j$ is omitted from the product.  Here $(b_1 \cdots b_m)$ is the principal ideal generated by $b_1 \cdots b_m$.  The $b_i$'s in this particular U-factorization above will be referred to as \emph{essential divisors}. The $a_i$'s in this particular U-factorization above will be referred to as \emph{inessential divisors}.  A U-factorization is said to be \emph{trivial} if there is only one essential divisor.
\\
\indent  A \emph{$\tau$-U-factorization} of a non-unit $a\in R$ is a U-factorization  $a=\lambda a_1 a_2 \cdots a_n \left\lceil b_1 b_2\cdots b_m \right\rceil$ for which $\lambda a_1 \cdots a_n b_1 \cdots b_m$ is also a $\tau$-factorization.  
\\
\indent Given a symmetric relation $\tau$ on $R^{\#}$, we say $R$ is \emph{$\tau$-U-refinable} if for every $\tau$-U-factorization of any non-unit $a\in U(R)$, $a=\lambda a_1 \cdots a_n \left\lceil b_1 \cdots b_m\right\rceil$, any $\tau$-U-factorization of an essential divisors, $b_i=\lambda' c_1 \cdots c_{n'} \left\lceil d_1 \cdots d_{m'}\right\rceil$ satisfies 
$$a=\lambda \lambda' a_1 \cdots a_n c_1 \cdots c_{n'}\left\lceil b_1 \cdots b_{i-1} d_1 \cdots d_{m'} b_{i+1} \cdots \ b_m\right\rceil$$ is a $\tau$-U-factorization.  
\\
\indent Let $\alpha \in \{$irreducible, strongly irreducible, m-irreducible, very strongly irreducible$\}$.  Let $a$ be a non-unit.  If $a=\lambda a_1 a_2 \cdots a_n \left\lceil b_1 b_2\cdots b_m \right\rceil$ is a $\tau$-U-factorization, then this factorization is said to be a \emph{$\tau$-U-$\alpha$-factorization} if it is a $\tau$-U-factorization and the essential divisors $b_i$ are $\tau$-$\alpha$ for $1 \leq i \leq m$.  
\\
\indent We now define the finite factorization properties using the $\tau$-U-factorization approach.  Let $\alpha \in \{$ irreducible, strongly irreducible, m-irreducible, unrefinably irreducible, very strongly irreducible $\}$ and let $\beta \in \{$associate, strongly associate, very strongly associate $\}$.  $R$ is said to be \emph{$\tau$-U-$\alpha$} if for all non-units $a\in R$, there is a $\tau$-U-$\alpha$-factorization of $a$.  $R$ is said to satisfy \emph{$\tau$-U-ACCP} (ascending chain condition on principal ideals) if every properly ascending chain of principal ideals $(a_1) \subsetneq(a_2) \subsetneq \cdots $ such that $a_{i+1}$ is an essential divisor in some $\tau$-U-factorization of $a_i$, for each $i$ terminates after finitely many principal ideals.  $R$ is said to be a \emph{$\tau$-U-BFR} if for all non-units $a\in R$, there is a bound on the number of essential divisors in any $\tau$-U-factorization of $a$. 
\\
\indent $R$ is said to be a \emph{$\tau$-U-$\beta$-FFR} if for all non-units $a\in R$, there are only finitely many $\tau$-U-factorizations up to rearrangement of the essential divisors and $\beta$.  $R$ is said to be a \emph{$\tau$-U-$\beta$-WFFR} if for all non-units $a\in R$, there are only finitely many essential divisors among all $\tau$-U-factorizations of $a$ up to $\beta$.  $R$ is said to be a \emph{$\tau$-U-$\alpha$-$\beta$-divisor finite (df) ring} if for all non-units $a\in R$, there are only finitely many essential $\tau$-$\alpha$ divisors up to $\beta$ in the $\tau$-U-factorizations of $a$. 
\\
\indent $R$ is said to be a \emph{$\tau$-U-$\alpha$-HFR} if $R$ is $\tau$-U-$\alpha$ and for all non-units $a\in R$, the number of essential divisors in any $\tau$-U-$\alpha$-factorization of $a$ is the same.  $R$ is said to be a \emph{$\tau$-U-$\alpha$-$\beta$-UFR} if $R$ is a $\tau$-U-$\alpha$-HFR and the essential divisors of any two $\tau$-U-$\alpha$-factorizations can be rearranged to match up to $\beta$.  
\\
\indent The following diagram summarizes the main results from from \cite[Theorem 4.3 and Theorem 4.4]{Mooney2} where $\approx$ represents $R$ being strongly associate, and $\dagger$ represents $R$ is $\tau$-U-refinable:

$$
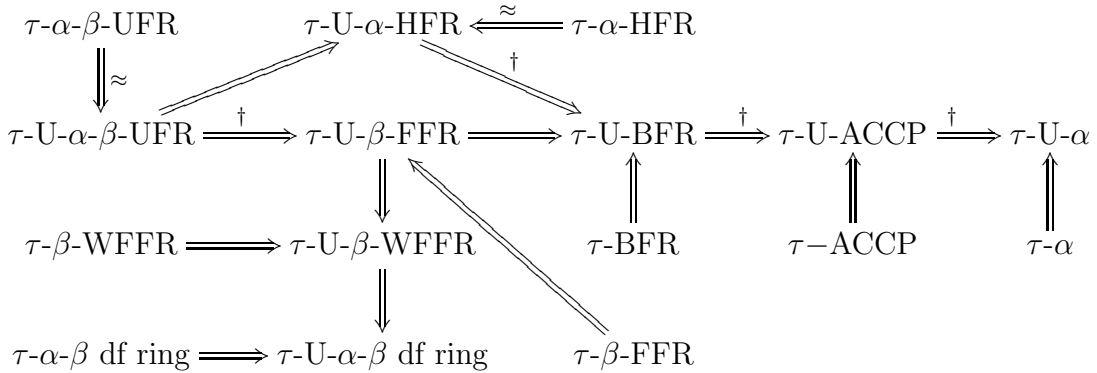
\xymatrix{
    \tau \text{-} \alpha \text{-} \beta \text{-UFR} \ar@{=>}^{\approx}[d]       &        \tau \text{-U-} \alpha\text{-HFR} \ar@{=>}^{\dagger}[dr]     &   \tau \text{-} \alpha \text{-HFR} \ar@{=>}_{\approx}[l]          &                  &                 \\
\tau \text{-U-} \alpha \text{-} \beta \text{-UFR} \ar@{=>}[ur] \ar@{=>}^{\dagger}[r]  & \tau \text{-U-} \beta \text{-FFR} \ar@{=>}[r] \ar@{=>}[d]  & \tau \text{-U-BFR} \ar@{=>}[r]^{\dagger}& \tau\text{-U-ACCP} \ar@{=>}^{\dagger}[r]& \tau\text{-U-}\alpha\\
 \tau \text{-} \beta \text{-WFFR} \ar@{=>}[r]           & \tau \text{-U-} \beta \text{-WFFR} \ar@{=>}[d]                     &   \tau \text{-} \text{BFR} \ar@{=>}[u]           &  \tau{-} \text{ACCP} \ar@{=>}[u]               &       \tau \text{-} \alpha \ar@{=>}[u]           \\
 \tau\text{-}\alpha \text{-} \beta \text{ df ring}   \ar@{=>}[r]        & \tau\text{-}\text{U-}\alpha \text{-} \beta \text{ df ring} & \tau\text{-}\beta \text{-FFR} \ar@{=>}[uul]
            }$$

\section{$\tau$-Regular Factorization} \label{sec: regular}
\indent The primary benefit of looking at the factorization of the regular elements is that for regular elements, all of the associate relations coincide.  That is, let $a,b \in \reg$, then $a\sim b$ implies $a\cong b$.  Suppose $a=rb$.  Neither $a$ nor $b$ can be zero, or else they could not be regular elements since we assume $R$ has an identity which is not zero.  But $a\sim b$ implies there is an $s\in R$ such that $b=sa$.  Thus $a=rb=r(sa)=(rs)a$, but $a$ is regular, so $a(1-rs)=0$ implies $rs-1=0$ or $rs=1$, so $r \in U(R)$ as desired.  Another important consequence is that for a regular element, we always have $a \cong a$.  This means that for a regular, non-unit element $a\in \reg$, if $a$ is irreducible, then $a$ is very strongly irreducible.  As a consequence, for a regular, non-unit $a\in R$ we can simply refer to it as \emph{irreducible} without any ambiguity.  We will soon see that this simplifies matters considerably.
\subsection{$\tau$-Regular Factorization Definitions}\ 
\\
\indent  Let $\tau$ be a symmetric relation on $R^{\#}$.  A $\tau$-factorization, $a=\lambda a_1 \cdots a_n$ with $\lambda \in U(R)$, and $a_i \tau a_j$ for all $i \neq j$ is said to be a \emph{$\tau$-regular-factorization} or \emph{$\tau$-r-factorization} if $a \in \reg$.  Note that $a$ is regular if and only if $a_i$ is regular for each $1\leq i \leq n$. 
 
\begin{proposition}\label{prop: atoms} Let $R$ be a commutative ring with $1$ and let $\tau$ be a symmetric relation on $R^{\#}$.  Given $a \in \reg$, the following are equivalent.
\\
(1) For any $\tau$-regular-factorization, $a=\lambda a_1 \cdots a_n$, we have $a\sim a_i$ for some $1 \leq i \leq n$.\\
(2) For any $\tau$-regular-factorization, $a=\lambda a_1 \cdots a_n$, we have $a\approx a_i$ for some $1 \leq i \leq n$.\\
(3) For any $\tau$-regular-factorization, $a=\lambda a_1 \cdots a_n$, we have $a\sim a_i$ for all $1 \leq i \leq n$.\\
(4) The only $\tau$-regular factorizations of $a$ are of the form $a=\lambda (\lambda^{-1}a)$.\\
(5) $a\cong a$ and for any $\tau$-regular-factorization, $a=\lambda a_1 \cdots a_n$, we have $a\cong a_i$ for some $1 \leq i \leq n$.\\
\end{proposition}
\begin{proof}
(5) $\Rightarrow$ (4) Suppose $a=\lambda a_1 \cdots a_n$ is a $\tau$-regular factorization with $n\geq 2$.  Then by hypothesis $a \cong a_i$ for some $1 \leq i \leq n$.  Then 
$$a=(\lambda a_1 \cdots a_{i-1} \widehat{a_i} a_{i+1} \cdots a_n)a_i$$
implies that $(\lambda a_1 \cdots a_{i-1} \widehat{a_i} a_{i+1} \cdots a_n)$ is a unit.  Hence the factorization was a trivial factorization to begin with.
\\
\indent (4) $\Rightarrow$ (3) is immediate.  After noting that any divisor of a regular element must be regular and hence $\sim, \approx$ and $\cong$ coincide, it is clear that (3) $\Rightarrow$ (2) and $(2) \Rightarrow (1)$.
\\
\indent (1) $\Rightarrow$ (5) Since $a$ is regular by hypothesis, $a\cong a$ and again $\sim, \approx$ and $\cong$ coincide on any divisors of a regular element, completing the proof.
\end{proof}
\indent We say that a non-unit, $a\in \reg$ is \emph{$\tau$-$r$-irreducible} or \emph{a $\tau$-$r$-atom} if $a$ satisfies any of the above equivalent conditions.  We say \emph{$R$ is $\tau$-$r$-atomic} if for all $a\in $Reg$(R)^{\#}$, there is a $\tau$-r-factorization into $\tau$-r-irreducible elements.  $R$ satisfies \emph{$\tau$-$r$-ACCP} if for every chain of principal ideals generated by regular elements $(a_1)\subsetneq (a_2) \subsetneq \cdots (a_i) \subsetneq \cdots$ with $a_{i+1}$ occurring as a $\tau$-divisor in some $\tau$-r-factorization of $a_i$ for all $i$ becomes stationary.
\\
\indent $R$ is a \emph{$\tau$-$r$-half factorization ring (HFR)} if (1) $R$ is $\tau$-$r$-atomic and (2) if $\lambda a_1 \cdots a_m=\mu b_1\cdots b_n$ are two $\tau$-$r$-atomic $\tau$-factorizations implies that $m=n$.  $R$ is said to be a \emph{$\tau$-r-unique factorization ring (UFR)} if $R$ is a $\tau$-$r$-HFR and there is a rearrangement of any two $\tau$-r-atomic factorizations as above such that $a_i\sim b_i$ for all $1 \leq i \leq n=m$.  We define the \emph{$\tau$-regular-elasticity} as $\tau$-$r$-$\rho(R)=$ sup$\{\rho(a)\mid a\in $Reg$(R)^{\#}\}$ where $\rho(a)=$ sup$\{\frac{m}{n} \mid \lambda a=a_1 \cdots a_m=\mu b_1 \cdots b_n$ are $\tau$-atomic-factorizations $\}$.  Then it is clear that $R$ is a $\tau$-$r$-HFR if and only if $R$ is $\tau$-atomic and $\tau$-$r$-$\rho(R)$=1.  
\\
\indent $R$ is said to be a \emph{$\tau$-$r$-bounded factorization domain (BFR)} if for every $a\in$Reg$(R)$ there exists a natural number $N_r(a)$ such that for all $\tau$-r-factorizations $a=\lambda a_1 \cdots a_n$, we have $n\leq N_r(a)$.  $R$ is said to be a \emph{$\tau$-$r$-irreducible-divisor-finite ring (idf ring)} if each $a\in $Reg$(R)^{\#}$ has at most a finite number of non-associate $\tau$-irreducible $\tau$-divisors.  $R$ is said to be a  \emph{$\tau$-$r$-finite factorization ring (FFR)} if for every $a\in$Reg$(R)^{\#}$, $a$ has only a finite number (up to order and associates) of $\tau$-factorizations.  $R$ is said to be a \emph{$\tau$-$r$-weak finite factorization ring (WFFR)} if for every $a\in$\text{Reg}$(R)^{\#}$ there are only a finite number of non-associate $\tau$-divisors.

\subsection{$\tau$-Regular Factorization Results} \label{sec: tau-regular-results}
\begin{proposition}\label{thm:reg-ffr} Let $R$ be a commutative ring with $1$.  Let $\tau$ be a symmetric relation on $R^{\#}$ with $\tau$ refinable, then the following are equivalent.
\\(1) $R$ is a $\tau$-$r$-FFR.
\\(2) $R$ is a $\tau$-$r$-WFFR.
\\(3) $R$ is a $\tau$-$r$-atomic $\tau$-$r$-idf ring.
\\(4) $R$ is $\tau$-$r$-atomic and each $a\in$Reg$(R)^{\#}$, $a$ has only finitely many $\tau$-$r$-atomic $\tau$-factorizations up to order and associates.
\\(5) For all $a\in$Reg$(R)^{\#}$, there are only finitely many $b\in$Reg$(R)^{\#}$ up to associate such that $b$ occurs as a $\tau$-factor in a $\tau$-r-factorization of $a$.
\\(6) For all $a\in$Reg$(R)^{\#}$, $(a)$ is contained in only finitely many principal ideals $(b)$ where $b\in\text{Reg}(R)^{\#}$ such that $b$ occurs as a $\tau$-factor in a $\tau$-r-factorization of $a$.
\\(7) For all $a\in$Reg$(R)^{\#}$, there are only finitely many $b\in$Reg$(R)^{\#}$ up to associate such that $b\mid_\tau a$.
\\(8) For all $a\in$Reg$(R)^{\#}$, $(a)$ is contained in only finitely many principal ideals $(b)$ where $b\in\text{Reg}(R)^{\#}$ such that $b\mid_\tau a$.
\end{proposition}
\begin{proof}  (1) $\Rightarrow$ (2) Let $R$ be a $\tau$-$r$-FFR and $a\in$\text{Reg}$(R)^{\#}$, then there are only a finite number of $\tau$-factorizations (up to order and associate), each of these is of finite length.  Hence, since every $\tau$-divisor of $a$ must be among these up to associate, $R$ is a $\tau$-$r$-WFFR.
\\
\indent (2) $\Rightarrow$ (3) Let $R$ be a $\tau$-$r$-WFFR and $a\in$\text{Reg}$(R)^{\#}$.  If $a$ has a finite number of $\tau$-divisors, then certainly it has a finite number of irreducible $\tau$-divisors, so it suffices to show $a$ has a $\tau$-$r$-atomic factorization.  We instead show the stronger condition, that $R$ satisfies $\tau$-$r$-ACCP, that is any chain of principal ideals generated by regular elements $(a_0) \subsetneq (a_1) \subsetneq \cdots \subset (a_i) \subsetneq \cdots$ with $a_{i+1}$ occurring as a $\tau$-factor in a $\tau$-r-factorization of $a_i$ and $a_i\in$\text{Reg}$(R)^{\#}$ for all $i$ comes to a halt.  Suppose there is an infinite chain, but then each $a_i$ is a $\tau$-divisor of $a_0$ and none of them are associate since each containment is proper, so we would have an infinite number of non-associate $\tau$-$r$-divisors contradicting the fact that $R$ is a $\tau$-$r$-WFFR (note: we use strongly here that $\tau$ is refinable to ensure that at each step we retain a $\tau$-factorization).
\\
\indent (3) $\Rightarrow$ (1) This proof is similar to \cite[Thm 5.1]{anderson90}. Let $R$ be a $\tau$-$r$-atomic $\tau$-$r$-idf ring and $x\in\text{Reg}(R)^{\#}$.  Let $x_1, \cdots, x_n$ be the $\tau$-$r$-irreducible $\tau$-factors of $x$, in particular they are all regular elements of $R$.  Suppose that in a $\tau$-factorization of $x$, $x=\lambda x_1^{s_1}\cdots x_n^{s_n}$, we always have $0 \leq s_i \leq N_i$ for each $1 \leq i \leq n$.  Then there is a bound on the number of non-associate factors of $x$.  So we suppose that this is not the case.  There must then be some $s_i$ which is not bounded, we assume it is the first one $s_1$.  Hence for each $k \geq 1$, we can write $x=\lambda_k x_1^{s_{k_1}} \cdots x_n^{s_{k_n}}$, where $\lambda_k\in U(R)$ and $s_{1_1} < s_{2_1} < s_{3_1} <\cdots$.  Suppose that in this set of factorizations $\{s_{k_i}\}$ is bounded for each $i$ with $1<i\leq n$.  Then since there are only finitely many choices for $s_{k_2}, \cdots, s_{k_n}$ we must have $s_{k_2}=s_{j_2}, \cdots, s_{k_n}=s_{j_n}$ for some $j>k$.  But then $\lambda_j x_1^{s_{j_1}} \cdots x_n^{s_{j_n}}=x=\lambda_k x_1^{s_{k_1}} \cdots x_n^{s_{k_n}}$, but since each $x_i$ is regular, we can cancel to get $\lambda_j x_1^{s_{j_1}}=\lambda_k x_1^{s_{k_1}}$, where $s_{j_1}>s_{k_1}$, but then $x_1$ would be a unit, a contradiction.  
\\
\indent Thus we must have some set $\{s_{k_i}\}$ for a fixed $i$ with $1<i\leq n$ is unbounded, say for $i=2$.  By taking subsequences at each stage, we may assume that $s_{1_1} < s_{2_1} < s_{3_1} < \cdots$ and $s_{1_2} < s_{2_2} < s_{3_2} < \cdots$.  Continuing in this manner, we may assume for each $1 \leq i \leq n$ that $s_{1_i} < s_{2_i} < s_{3_i} < \cdots $.  But then we would have $\lambda_1 x_1^{s_{1_1}} \cdots x_n^{s_{1_n}}=x=\lambda_2 x_1^{s_{2_1}} \cdots x_n^{s_{2_n}}$ where $s_{1_i}<s_{2_i}$, a contradiction as again, we would have $x_i$ must be units after cancellation, which is impossible.
\\
\indent (1) $\Rightarrow$ (4) This is clear as we have already seen that a $\tau$-$r$-FFR is $\tau$-$r$-atomic and a $\tau$-$r$-atomic factorization is certainly a $\tau$-$r$-factorization, so there must be a finite number of $\tau$-$r$-atomic factorizations up to order and associate for every $a\in\text{Reg}(R)^{\#}$. 
\\
\indent (4) $\Rightarrow$ (3) Let $a\in\text{Reg}(R)^{\#}$, then there are a finite number of $\tau$-$r$-atomic factorizations, each has a finite number of $\tau$-$r$-atomic factors, so the collection of $\tau$-$r$-atomic divisors is finite, so $R$ is a $\tau$-$r$-atomic $\tau$-$r$-idf ring.
\\
\indent (5), (6) are restatements of (2) and their equivalence is immediate. Furthermore, (5) and (7) (resp. (6) and (8)) are seen to be equivalent after noting that for $b\in \reg$, $a\mid_\tau b$ implies there is some $\tau$-factorization $b=\lambda a a_1 \cdots a_n$, but since $b$ is regular and the set of regular elements is saturated, every $\tau$-factor must be regular so this is really a $\tau$-factorization.
\end{proof}
\begin{theorem} \label{thm: reg} Let $R$ be a commutative ring with $1$, with $\tau$ a symmetric relation on $R^{\#}$.  We have the following.
\\(1) $R$ is a $\tau$-$r$-UFR implies $R$ is a $\tau$-$r$-HFR.
\\(2) For $\tau$ refinable, $R$ is a $\tau$-$r$-HFR implies $R$ is a $\tau$-$r$-BFR.
\\(3) For $\tau$ refinable, $R$ is a $\tau$\text{-}$r$\text{-UFR} implies $R$ is a $\tau$-$r$-FFR.
\\(4) $R$ is a $\tau$-$r$-FFR implies $R$ is a $\tau$-$r$-BFR.
\\(5) For $\tau$ refinable, $R$ is a $\tau$-$r$-BFR implies $R$ satisfies $\tau$-$r$-ACCP.
\\(6) For $\tau$ refinable, $R$ satisfies $\tau$-$r$-ACCP implies $R$ is $\tau$-$r$-atomic.
\end{theorem}
\begin{proof}
(1) This is immediate from the definition.
\\
\indent (2) Let $R$ be a $\tau$-$r$-HFR.  Suppose $a=\lambda a_1 \cdots a_n$ is a $\tau$-r-atomic factorization.  We claim $N_r(a)=n$.  Let $a=\mu b_1 \cdots b_m$ be a $\tau$-r-factorization of $a$.  Since $R$ is $\tau$-r-atomic, we can find $\tau$-r-atomic factorizations for $b_i$ for each $1\leq i \leq m$.  We have assumed $\tau$ to be refinable, so we can replace each $b_i$ with the corresponding $\tau$-r-atomic factorization and collect the units in the front of the factorization and retain a $\tau$-r-factorization which is $\tau$-atomic and thus must have length $n$.  The refinement process can only increase the length of the factorization, so the length of the original factorization is no longer than $n$, proving the claim.
\\
\indent (3) We show for $\tau$-refinable, $R$ a $\tau$-$r$-UFR, $R$ is a $\tau$-$r$-atomic $\tau$-$r$-idf-ring which has been shown in Theorem \ref{thm:reg-ffr} to be equivalent to being a $\tau$-$r$-FFR.  $R$ being $\tau$-$r$-factorial gives us $\tau$-$r$-atomic for free.  Furthermore, any $\tau$-atomic factorization of $a\in$Reg$(R)^{\#}$ has the same length, say $n$ and can be reordered so that the associates match up.  This tells us there are precisely $n$ $\tau$-irreducible divisors of $a$ up to associate, hence $R$ is a $\tau$-$r$-idf-ring.  
\\
\indent (4) Suppose $R$ is a $\tau$-$r$-FFR, by definition, we know $R$ is $\tau$-$r$-atomic.  Now, let $a\in$Reg$(R)^{\#}$, let $S$ be the finite set of all $\tau$-atomic factors of $a$.  Set $N(a)=|S|$.  Let $a=\lambda a_1 \cdots a_n$ be a $\tau$-atomic factorization of $a$, then $a_i\in S$ for all $i$, but then $\{a_i\}_{i=1}^{n} \subseteq S$ and hence is finite and $n\leq N(a)=|S|$ as desired, so $R$ is a $\tau$-$r$-BFR.
\\
\indent (5) Let $R$ be a $\tau$-$r$-BFR, and we suppose for a moment that $R$ does not satisfy $\tau$-$r$-ACCP.  There must exist and infinite sequence $\{a_i\}_{i=1}^{\infty}\subseteq$Reg$(R)^{\#}$ such that $a_{n+1}\mid_{\tau} a_n$, but $a_{n+1}\not \sim a_n$ for all $n \geq 1$.  Let $a_n=\lambda_{n+1}r_{n+11}\cdots r_{n+1s_{n+1}}a_{n+1}$ be a $\tau$ factorization of $a_n$ for all $n \geq 1$.  But then we have
$$a_1=\lambda_2r_{2_1}\cdots r_{2_{s_2}}a_2=\lambda_2 r_{2_1}\cdots r_{2_{s_2}}\lambda_3 r_{3_1}\cdots r_{3_{s_3}}a_3=\cdots$$
is a $\tau$ factorization (note we use $\tau$ refinable here).  Furthermore, each of these factorizations can be refined into $\tau$-atomic elements, and it will still be a $\tau$-factorization the length of which $L_{\tau}(a_1)\geq s_2 +s_3 + \cdots s_n + 1 \geq n$ which shows we can find arbitrarily large $\tau$-atomic factorizations of $a_1$ which contradicts the fact the $R$ is a $\tau$-$r$-BFR.
\\
\indent (6) Let $R$ satisfy $\tau$-$r$-ACCP, but suppose that $R$ is not $\tau$-$r$-atomic.  Then there exists $a\in$Reg$(R)^{\#}$ with no $\tau$-factorization into $\tau$-atoms.  $a$ itself cannot be a $\tau$-atom, so say $a=\lambda a_1 \cdots a_n$ is a $\tau$-factorization with $n >1$.  Now again some $a_i$ must not be a product of $\tau$-atoms, or with $\tau$ refinable, we could find a $\tau$-atomic factorization, say it is $a_1$.  Then $a_1 \mid_{\tau} a$ and $a_1\not \sim a$ put $b_1=a_1$.  Then $a_1$ must have a $\tau$-factorization $a_1=\lambda_2 a_{2_1}\cdots a_{2_{n_2}}$ where $n_2 >1$.  Again, one of the $\tau$-factors, say $a_{2_1}$ cannot be a $\tau$-product of $\tau$-atoms.  Here $a_{2_1} \mid_{\tau}a_1=b_1$ and $a_{21}\not \sim a_1$.  Put $b_2=a_{2_1}$.  Continuing in this fashion, we obtain a sequence $\{b_i\}_{i=1}^{\infty}$ of elements of Reg$(R)^{\#}$ such that $b_{n+1} \mid_{\tau}b_n$ but $b_{n+1} \not \sim b_n$ for every $n \geq 1$.  This contradicts $R$ satisfying $\tau$-$r$-ACCP.
\end{proof}
The following diagram summarizes our result where $\nabla$ represents $\tau$ being refinable.
$$
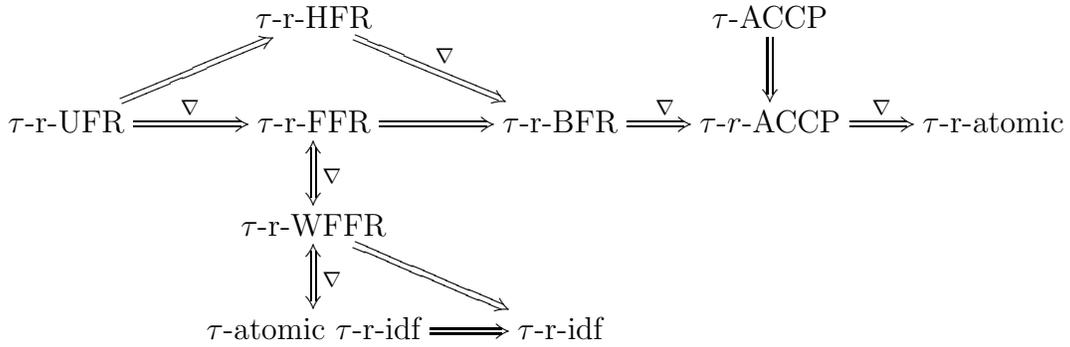
\xymatrix{
            &        \tau\text{-r}\text{-HFR} \ar@{=>}^{\nabla}[dr]     &             &         \tau \text{-ACCP}\ar@{=>}[d]        &                 \\
\tau\text{-r}\text{-UFR} \ar@{=>}[ur] \ar@{=>}^{\nabla}[r]& \tau\text{-r}\text{-FFR} \ar@{=>}[r] & \tau\text{-r}\text{-BFR} \ar@{=>}[r]^{\nabla}& \tau\text{-}r\text{-ACCP} \ar@{=>}^{\nabla}[r]& \tau\text{-r}\text{-atomic}\\
            &        \tau\text{-r}\text{-WFFR} \ar@{<=>}_{\nabla}[u]  \ar@{=>}[dr]   &             &            &                 \\
            &        \tau \text{-atomic  }\tau\text{-r}\text{-idf} \ar@{<=>}_{\nabla}[u]  \ar@{=>}[r]    & \tau\text{-r}\text{-idf}          &            &                }
$$
\section{$\treg$-Factorizations} \label{sec: tau-regular}
\indent In this section, we study another approach which could have be used to extend $\tau$-factorization to commutative rings with zero-divisors using regular factorizations.  In Section \ref{sec: regular}, we decided to only consider factorizations of the regular elements.  In other words, we chose to restrict the elements we attempt to factor to the regular elements of a commutative ring $R$.  We could have instead chosen to restrict the relation $\tau$ itself.  This gives us the benefit of not completely ignoring a possible large number of zero-divisors in the ring $R$, but at the cost of choosing a less natural relation $\tau$.  Moreover, it allows us to use much of the work done previously in \cite{Mooney} by just picking a different $\tau$ and keeping all of the original definitions the same.  It turns out that in many ways, either choice is fine and we end up at the same place anyway.  Studying this will be the motivation of this section.
\\
\indent Let $R$ be a commutative ring with $1$ and $\tau$ a symmetric relation on $R^{\#}$.  Then we define a new relation
$$\treg:=\tau \cap \left(\reg \times \reg \right).$$
\indent We may now pursue the $\tau$-factorizations using the approach from \cite{Mooney} and look at factoring all the non-units in $R$ instead of just the regular elements.  There is certainly a very close relationship between $\treg$-factorizations and $\tau$-regular factorizations; however, there are a few subtle differences that cause some problems, especially with the definition of $\treg$-very strongly atomic elements.  In, \cite{Mooney}, the author insisted that part of $a$ being $\tau$-very strongly atomic was that $a\cong a$.  
\\
\indent The fact that the very strongly associate relation need not be reflexive is the main reason there is not a perfect correspondence between the two approaches.  We will see that $\treg$-factorizations are simply very poorly behaved when it comes to $\treg$-very strong atoms and rearrangement up to very strong associates.  On the bright side, the $\tau$-unrefinably irreducible element introduced in \cite{Mooneycomplete} will also behave quite nicely here.
\\
\indent Of course any non-trivial idempotent element, $e$, is a zero-divisor since $e(e-1)=0$.  Furthermore, since $e=e^2=e\cdot e$, with $e$ not a unit, we see that $e\not \cong e$.  This means that $e$ is not very strongly atomic for any non-trivial idempotent element.  On the other hand, since every non-trivial $\treg$-factorization consists of a product of regular elements, we can have no non-trivial $\treg$-factorizations of $e$.  This means the only $\treg$-factorizations of any zero-divisor, in particular $e$, are the trivial factorizations.  Unfortunately, in the case of a non-trivial idempotent, $e$, this means $e$ is not a $\tau$-very strong atom, and will never have a $\treg$-very strongly atomic factorization.  We demonstrate this in the following example.
\\
\begin{example} Let $K$ be an infinite field.  $R=K \times K$ with $\tau=R^{\#} \times R^{\#}$. \end{example}
We consider the element $(1,0)\in Z(R)$.  This ring has only elements which are strongly associate to idempotent elements and units.  So the set of non-unit regular elements is empty and our ring is vacuously a $\tau$-r-UFR.  On the other hand, we have $(1,0)=(\mu^{-1}, 1)(\mu,0)$, for any unit $\mu \in K^*$,  is the only type of $\treg$-factorization of $(1,0)$, yet none of these are $\treg$-very strongly atomic factorizations.  The problem is that $(\mu, 0) \not \cong (\mu,0)$ since we have $(\mu,0)=(1,0)(\mu, 0)$ and $(1,0)$ is not a unit.  This shows we can have a $\tau$-r-UFR which is not even $\treg$-atomic.  Moreover, each of these factorizations is non-very strongly associate.  Let $\mu, \lambda \in K^*$.  Then $(1,0)=(\mu^{-1}, 1)(\mu,0)=(\lambda^{-1}, 1)(\lambda,0)$ are two $\treg$-factorizations of $(1,0)$, but $(\mu,0)=(\mu \lambda^{-1},0)(\lambda,0)$ with $(\mu \lambda^{-1},0)$ not a unit shows $(\mu,0) \not \cong (\lambda, 0)$.  Since $K$ is infinite, there are infinitely many $\treg$-factorizations of $(1,0)$, none of which are very strongly associate.
This leads us to the following results.
\begin{lemma}\label{lem: treg-nontrivial} Let $R$ be a commutative ring with $1$ and let $\tau$ be a symmetric relation on $R^{\#}$.  Let $\treg:=\tau \cap \left(\reg \times \reg \right)$.  The collection of non-trivial $\tau$-regular-factorizations and non-trivial $\treg$-factorizations coincide.
\end{lemma}
\begin{proof} Let $a=\lambda a_1 \cdots a_n$ be a non-trivial $\tau$-regular factorization.  Then $a \in \reg$ by definition of $\tau$-regular factorization, and $a_i \tau a_j$ for all $i \neq j$.  Since $a$ is regular, and the set of regular elements is saturated, we have $a_i \mid a \in \reg$ for each $1 \leq i \leq n$, we know that $a_i \in \reg$ for each $1 \leq i \leq n$.  This means $a_i \treg a_j$ for each $i \neq j$.  Thus $a=\lambda a_1 \cdots a_n$ is a $\treg$-factorization.\\
\indent Conversely, suppose $a=\lambda a_1 \cdots a_n$ is a non-trivial $\treg$-factorization.  Then $a_i \treg a_j$ for each $i \neq j$.  This means $a_i \tau a_j$ and $a_i, a_j \in \reg$.  In particular, since $n \geq 2$, we can conclude that $a_1 a_2 \cdots a_n$ is a product of regular elements, so $a \in \reg$.  This means $a=\lambda a_1 \cdots a_n$ is a $\tau$-regular-factorization.
\end{proof}

\begin{theorem} \label{thm: treg-reg} Let $R$ be a commutative ring with $1$ and let $\tau$ be a symmetric relation on $R^{\#}$.  Let $\treg:=\tau \cap \left(\reg \times \reg \right)$.  For $a\in \reg$, the following are equivalent.
\\
(1) $a$ is a $\tau$-regular-atom.
\\
(2) $a$ is a $\treg$-atom.
\\
(3) $a$ is a $\treg$-strong atom.
\\
(4) $a$ is a $\treg$-m-atom.
\\
(5) $a$ is a $\treg$-unrefinable atom.
\\ 
(6) $a$ is a $\treg$-very strong atom.
\end{theorem}
\begin{proof} When we consider Theorem \ref{thm: unrefinable}, it suffices to show that $(2) \Rightarrow (6)$ and then we show that $(1) \Leftrightarrow (5)$.  Let $a \in \reg$, be a $\treg$-atom.  Since $a \in \reg$, we have $a\cong a$ since $a=ra$ implies $r=1$.  Furthermore, if $a=\lambda a_1 \cdots a_n$ is a $\treg$-factorization of $a$, then $a \sim a_i$ for some $1 \leq i \leq n$.  Since $a \in \reg$, $a \cong a_i$ and we have shown that $a$ is a $\treg$-very strongly atom.
\\
\indent (1) $\Leftrightarrow$ (5) In light of Lemma \ref{lem: treg-nontrivial}, $a$ has a non-trivial $\tau$-regular factorization if and only if $a$ has a non-trivial $\treg$-factorization.
\end{proof}
\begin{corollary} \label{cor: treg-atomic-factorizations} Let $R$ be a commutative ring with $1$ and let $\tau$ be a symmetric relation on $R^{\#}$.  Let $\treg:=\tau \cap \left(\reg \times \reg \right)$. Let $\alpha \in \{$ atomic, strongly atomic, m-atomic, unrefinably atomic $\}$.  Let $a\in \reg$ be a non-unit, then $a=\lambda a_1 \cdots a_n$ is a $\treg$-$\alpha$-factorization if and only if $a=\lambda a_1 \cdots a_n$ is a $\tau$-regular-atomic factorization.
\end{corollary}
\begin{proof} This is immediate from what we have shown in Theorem \ref{thm: treg-reg}.
\end{proof}
\begin{theorem}\label{thm: treg-zd}Let $R$ be a commutative ring with $1$ and let $\tau$ be a symmetric relation on $R^{\#}$.  Let $\treg:=\tau \cap \left(\reg \times \reg \right)$.  If $a\in Z(R)$, then following hold.
\\
(1) $a$ is a $\treg$-atom.
\\
(2) $a$ is a $\treg$-strong atom.
\\
(3) $a$ is a $\treg$-m-atom.
\\
(4) $a$ is a $\treg$-unrefinable atom.
\end{theorem}
\begin{proof}
By Theorem \ref{thm: unrefinable}, it suffices to show, for $a \in Z(R)$, $(1) \Rightarrow (4)$.  Let $a$ be a $\treg$-atom, and suppose $a=\lambda a_1 \cdots a_n$ is a non-trivial $\treg$-factorization.  This implies $n\geq 2$, and therefore $a_i \treg a_j$ for each $i \neq j$.  In particular, $a_i\in \reg$ for all $1 \leq i \leq n$.  This means $a$ is a product of regular elements and is therefore regular, a contradiction.
\end{proof}

\begin{theorem}\label{thm: treg-atomic}Let $R$ be a commutative ring with $1$ and let $\tau$ be a symmetric relation on $R^{\#}$.  Let $\treg:=\tau \cap \left(\reg \times \reg \right)$. The following are equivalent.
\\
(1) $R$ is $\tau$-regular-atomic.
\\
(2) $R$ is a $\treg$-atomic.
\\
(3) $R$ is a $\treg$-strongly atomic.
\\
(4) $R$ is a $\treg$-m-atomic.
\\
(5) $R$ is $\treg$-unrefinably atomic.
\end{theorem}
\begin{proof} Let $a$ be a non-unit in $R$.  Then $a \in Z(R)$ or $a\in \reg$.  If $a\in Z(R)$, we apply Theorem \ref{thm: treg-zd} to see that $a$ itself is $\treg$-atomic, $\treg$-strongly atomic, $\treg$-m-atomic, and $\treg$-unrefinably atomic and $a=1\cdot a$ is a $\treg$-atomic, $\treg$-strongly atomic, $\treg$-m-atomic, and $\treg$-unrefinably atomic factorization of $a$.  For $R$ to be a $\tau$-regular-atomic ring, we need only check the regular elements for $\tau$-regular atomic factorizations.  If $a\in \reg$, we apply Corollary \ref{cor: treg-atomic-factorizations} to see that $a$ has a $\tau$-regular-atomic factorization if and only if $a$ has a $\treg$-atomic (resp. $\treg$-strongly atomic, $\treg$-m-atomic, $\treg$-unrefinably atomic) factorization.  This completes the equivalence since we have checked both the zero-divisors as well as the regular elements.  
\end{proof}
\begin{lemma} \label{lem: trivial} Let $R$ be a commutative ring with $1$ and let $\tau$ be a symmetric relation on $R^{\#}$.  Let $a=\lambda (\lambda^{-1}a)=\mu (\mu^{-1}a)$ be two trivial factorizations of $a$.  Then we have the following
\\
(1) $\lambda^{-1}a$ and $\mu^{-1}a$ are associates.
\\
(2) $\lambda^{-1}a$ and $\mu^{-1}a$ are strong associates.
\end{lemma}
\begin{proof} $(\mu^{-1} \lambda)(\lambda^{-1}a)=\mu^{-1}a$ with $(\mu^{-1} \lambda)\in U(R)$ proves $\lambda^{-1}a \approx \mu^{-1}a$.  If $\lambda^{-1}a \approx \mu^{-1}a$, then $\lambda^{-1}a \sim \mu^{-1}a$.  This proves both (2) and (1).
\end{proof}

\begin{remark}Given the above situation, $\lambda^{-1}a$ and $\mu^{-1}a$ need not be very strong associates.  For instance $R=\R \times \R$, 
$$(1,0)=(1,1)(1,0)=(-1,-1)(-1,0)$$
yet $(1,0) \not \cong (-1,0)$.
\end{remark}

\begin{theorem}\label{thm: treg-UFR} Let $R$ be a commutative ring with $1$ and let $\tau$ be a symmetric relation on $R^{\#}$.  Let $\treg:=\tau \cap \left(\reg \times \reg \right)$. Let $\alpha \in \{$ atomic, strongly atomic, m-atomic, unrefinably atomic $\}$ and $\beta \in \{$ associate, strongly associate $\}$.  Then we have the following.
\\
(1) $R$ satisfies $\tau$-regular-ACCP if and only if $R$ satisfies $\treg$-ACCP.
\\
(2) $R$ is a $\tau$-regular-UFR if and only if $R$ is a $\treg$-$\alpha$-$\beta$-UFR.
\\
(3) $R$ is a $\tau$-regular-HFR if and only if $R$ is a $\treg$-$\alpha$-HFR.
\\
(4) $R$ is a $\tau$-regular-BFR if and only if $R$ is a $\treg$-BFR.
\\
(5) $R$ is a $\tau$-regular-idf ring if and only if $R$ is a $\treg$-$\alpha$-$\beta$-df ring.
\\
(6) $R$ is a $\tau$-regular-atomic $\tau$-regular-idf ring if and only if $R$ is a $\treg$-$\alpha$, $\treg$-$\alpha$-$\beta$-df ring.
\\
(7) $R$ is a $\tau$-regular-WFFR if and only if $R$ is a $\treg$-$\beta$-WFFR.
\\
(8) $R$ is a $\tau$-regular-FFR if and only if $R$ is a $\treg$-$\beta$-FFR.
\\
\indent If $\tau$ is refinable, then $(6) \Leftrightarrow (7) \Leftrightarrow (8)$. 
\end{theorem}
\begin{proof}
(1) The statement that $(a) \subsetneq (a_1)$ with $a_1\mid_\tau a$ implies that $a=\lambda a_1 a_2 \cdots a_n$.  We notice here that $n \geq 2$ or else we would have $a=\lambda a_1$ or $a\approx a_1$ which implies $(a)=(a_1)$, a contradiction.  So these properly ascending chains yield non-trivial factorizations at each step.  Thus any properly ascending chain of principal ideals 
\begin{equation} 
\label{eq: ACCP}(a_1) \subsetneq (a_2) \subsetneq (a_3) \subsetneq \cdots 
\end{equation}
such that $a_{i+1} \mid_{\treg} a_i$ yields a $\tau$-regular factorization of $a_i$ with $a_{i+1}$ as a $\tau$-regular factor.  Conversely, any ascending chain as in \eqref{eq: ACCP} with $a_{i}$ regular for all $i$ and $a_{i+1}$ occurring as a $\tau$-factor in some $\tau$-regular factorization of $a_i$ yields a $\treg$-factorization of $a_i$ as well.  Hence $R$ fails to satisfy $\tau$-regular ACCP if and only if $R$ fails to satisfy $\treg$-ACCP, and the proof is complete.
\\
\indent (2) We know from Theorem \ref{thm: treg-atomic} that $R$ is $\tau$-regular-$\alpha$ if and only if $R$ is $\treg$-$\alpha$.  Let $a\in R$ be a non-unit.  If $a\in Z(R)$, we know from Theorem \ref{thm: treg-zd} that $a$ is $\treg$-$\alpha$.  Furthermore,  any trivial $\treg$-factorization of $a$ is unique up to $\beta$ by Lemma \ref{lem: trivial}.  For $R$ to be a $\tau$-regular UFR, we need only check the regular elements.  Let $a\in \reg$.  We know from Corollary \ref{cor: treg-atomic-factorizations}, for regular elements, $\tau$-atomic and $\treg$-$\alpha$-factorizations of $a$ coincide, so the uniqueness up to rearrangement and $\beta$ is immediate.
\\
\indent (3) By Theorem \ref{thm: treg-atomic}, $R$ is $\tau$-regular-$\alpha$ if and only if $R$ is $\treg$-$\alpha$.  If $a\in Z(R)$, then $a$ is $\treg$-$\alpha$ and has only trivial $\treg$-factorizations each of which has length $1$.  For $a\in \reg$, $\tau$-atomic and $\treg$-$\alpha$-factorizations of $a$ coincide by Corollary \ref{cor: treg-atomic-factorizations}, and the equivalence is clear.
\\
\indent (4) For $a\in Z(R)$, all $\treg$-factorizations are trivial and have length 1.  By Lemma \ref{lem: treg-nontrivial}, the set of non-trivial $\tau$-regular factorizations and $\treg$-factorizations coincide and the equivalence is apparent.
\\
\indent (5) If $a\in Z(R)$, $a$ itself is $\treg$-$\alpha$ and there is precisely one unique $\treg$-$\alpha$-divisor of $a$ up to $\beta$ since all trivial $\treg$-factorizations are $\beta$ from Lemma \ref{lem: trivial}.  If $a \in \reg$, then the set of $\tau$-regular atomic divisors and $\treg$-$\alpha$-divisors of $a$ are all regular and hence coincide by Theorem \ref{thm: treg-reg} so the equivalence is clear.
\\
\indent (6) This is simply (5) plus Theorem \ref{thm: treg-atomic}.
\\
\indent (7) For $a \in Z(R)$, the only $\treg$-divisors of $a$ are unit multiples of $a$, so there is only one $\treg$-divisor of $a$ up to $\beta$.  For $a\in \reg$, since the set of $\tau$-regular factorizations and the set of $\treg$-factorizations of $a$ are the same, the set of $\treg$-divisors and $\tau$-regular divisors coincide and are regular, so the associate relations also coincide.  Thus the equivalence follows.
\\
(8) \indent For $a \in Z(R)$, the only $\treg$-factorizations of $a$ are of the form $a=\lambda (\lambda^{-1}a)$, so there is only one $\treg$-factorization of $a$ up to $\beta$.  For $a\in \reg$, since the set of $\tau$-regular factorizations and the set of $\treg$-factorizations of $a$ are the same.  Moreover, the set of $\treg$-factors and $\tau$-regular factors coincide and are regular, hence the associate relations also coincide.  Thus the equivalence follows.
\end{proof}

\section{Relationship with Other Finite Factorization Properties} \label{sec: other props}
In this final section, we would like to demonstrate where the rings satisying the properties in the present article fit in with the various finite factorization properties already existing in the literature.  That is, would like to compare the $\tau$-regular and $\treg$-finite factorization properties with the regular factorization from \cite{Valdezleon2}, the $\tau$-finite factorization properties defined originally in \cite{Mooney} as well as the $\tau$-U-finite factorization properties defined in \cite{Mooney2}.  A note to the reader, many of these terms were defined in Section \ref{sec: prelim}.
\\
\indent The following theorem demonstrates that the $\tau$-finite factorization properties defined in \cite{Mooney} are stronger than the ones in the present article.
\begin{theorem} \label{thm: tau - tau-reg}  Let $R$ be a commutative ring with $1$ and let $\tau$ be a symmetric relation on $R^{\#}$.  Let $\alpha \in \{$atomic, strongly atomic, m-atomic, unrefinably atomic very strongly atomic$ \}$, $\beta \in \{$associate, strong associate, very strong associate$\}$.  Then we have the following:
\\
(1) If $R$ is a $\tau$-$\alpha$-$\beta$-UFR, then $R$ is a $\tau$-r-UFR.
\\
(2) If $R$ is a $\tau$-$\alpha$-HFR, then $R$ is a $\tau$-r-HFR.
\\
(3) If $R$ is a $\tau$-$\beta$-FFR, then $R$ is a $\tau$-r-FFR.
\\
(4) If $R$ is a $\tau$-$\beta$-WFFR, then $R$ is a $\tau$-r-WFFR.
\\
(5) If $R$ is a $\tau$-$\beta$-$\alpha$ df ring, then $R$ is a $\tau$-r idf ring.
\\
(6) If $R$ is a $\tau$-BFR, then $R$ is a $\tau$-r-BFR.
\\
(7) If $R$ satisfies $\tau$-ACCP, then $R$ satisfies $\tau$-r-ACCP.
\\
(8) If $R$ is $\tau$-$\alpha$, then $R$ is $\tau$-r-atomic.\\
This yields the following diagram where $\nabla$ represents $\tau$ is refinable.
$$
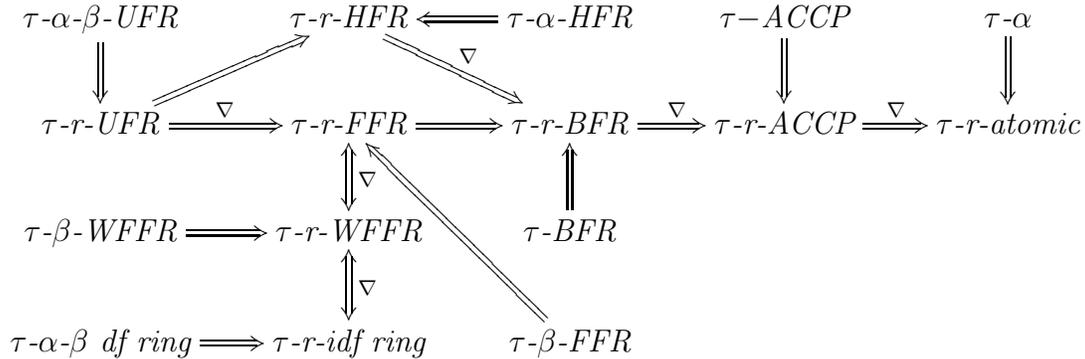
\xymatrix{
    \tau \text{-} \alpha \text{-} \beta \text{-UFR} \ar@{=>}[d]       &        \tau \text{-r-HFR} \ar@{=>}^{\nabla}[dr]     &   \tau \text{-} \alpha \text{-HFR} \ar@{=>}[l]          &   \tau{-} \text{ACCP} \ar@{=>}[d]   &         \tau \text{-} \alpha \ar@{=>}[d]                      \\
\tau \text{-r-UFR} \ar@{=>}[ur] \ar@{=>}^{\nabla}[r]  & \tau \text{-r-FFR} \ar@{=>}[r] \ar@{<=>}^{\nabla}[d]  & \tau \text{-r-BFR} \ar@{=>}[r]^{\nabla}& \tau\text{-r-ACCP} \ar@{=>}^{\nabla}[r] & \tau\text{-r-atomic} \\
 \tau \text{-} \beta \text{-WFFR} \ar@{=>}[r]           & \tau \text{-r-WFFR} \ar@{<=>}^{\nabla}[d]                     &   \tau \text{-} \text{BFR} \ar@{=>}[u]           &                     \\
 \tau\text{-}\alpha \text{-} \beta \text{ df ring}   \ar@{=>}[r]        & \tau\text{-}\text{r-idf ring} & \tau\text{-}\beta \text{-FFR} \ar@{=>}[uul]& &   
            }$$
\end{theorem}

\begin{proof}
\indent (8) Let $a\in \reg$.  Since $R$ is a $\tau$-$\alpha$, there is a $\tau$-$\alpha$-factorization of the form $a=\lambda a_1 \cdots a_n$.  Since $a\in \reg$, $a_i \in \reg$ for all $i$, by Proposition \ref{prop: atoms}, each of these factorizations is a $\tau$-r-atomic factorization of $a$, showing $R$ is $\tau$-r-atomic.
\\
\indent (2) (resp. (1)) Let $a$ be a regular non-unit element.  We have just seen that $R$ is $\tau$-r-atomic.  Given two $\tau$-r-atomic factorizations, $a=\lambda a_1\cdots a_n=\mu b_1 \cdots b_m$, this is also two $\tau$-$\alpha$-factorizations.  By assumption we have $m=n$ (resp. and there is a rearrangement so that $a_i \sim b_i$ for each $1\leq i \leq n$.)  This proves $R$ is a $\tau$-r-HFR (resp. $\tau$-r-UFR).
\\
\indent [(3)-(6)] Let $a \in \reg$.  For a regular element $a$, the set of $\tau$-r-factorizations and $\tau$-factorizations are identical, proving (3) and (6).  Similarly, since every divisor of a regular element is regular, the set of regular $\tau$-divisors is the same as the set of $\tau$-divisors, proving (4).  As in \ref{prop: atoms}, we know that the set of $\tau$-$\alpha$-divisors is the same as the set of $\tau$-r-atoms, proving (5).
\\ 
\indent (7) Suppose $(a_1) \subsetneq (a_2) \subsetneq \cdots$ is an chain of regular principal ideals such that $a_{i+1} \mid_{\tau} a_i$, then since $R$ satisfies $\tau$-ACCP, it must become stationary, proving (7).  
\end{proof}

\indent The following gives us a comparison of the regular factorization rings defined in \cite{Valdezleon2} with the rings defined in the current article.
\begin{theorem} \label{thm: reg-tau}Let $R$ be a commutative ring with $1$ and $\tau \subset \text{Reg(R)}^{\#} \times \text{Reg(R)}^{\#}$
\\(1) $R$ a $r$-BFR implies $R$ is a $\tau$-$r$-BFR
\\(2) $R$ a $r$-FFR implies $R$ is a $\tau$-$r$-FFR
\\(3) $R$ a $r$-WFFR implies $R$ is a $\tau$-$r$-WFFR
\\(4) $R$ satisfies $r$-ACCP implies $R$ satisfies $\tau$-$r$-ACCP.
\end{theorem}
\begin{proof}
(1) Let $R$ be a $r$-BFR, but suppose $R$ is not a $\tau$-$r$-BFR, then there exists a regular element $a\in \text{Reg(R)}^{\#}$ with $\tau$-factorizations of arbitrarily long length, but any $\tau$-factorization is certainly a factorization into regular elements, so this would contradict the fact that $R$ is a $r$-BFR.
\\
\indent (2) Let $R$ be a $r$-FFR, but suppose that $R$ is not a $\tau$-$r$-FFR.  We then have a regular element $a\in \text{Reg(R)}^{\#}$ that has an infinite number of $\tau$-$r$-factorizations up to rearrangement and associate, but again each of these is also $r$-factorization and are still unique up to rearrangement and associates which contradicts the fact that $R$ is a $r$-FFR.
\\
\indent (3) Let $a \in \reg$.  Every $\tau$-r-divisor divisor is a regular divisor of $a$, so there can be only finitely many up to associate.
\\
\indent (4) Suppose we have an infinite sequence $\{a_i\}_{i=1}^{\infty}$, $a_k\in$Reg$(R)^{\#}$ for all $k$ with $a_{n+1}\mid_{\tau}a_n$ but $a_{n+1}\not \sim a_n$ for all $n\geq 1$.  But then we still have $a_{n+1}\mid_{\tau}a_n$, $a_k\in$Reg$(R)^{\#}$ for all $k$ but $a_{n+1}\not \sim a_n$ so we contradict $r$-ACCP. Concluding the proof.
\end{proof}
\begin{corollary} The r-UFRs, r-FFRs, r-HFRs, r-BFRs as defined in \cite[Section 5]{Valdezleon3} satisfy $r$-ACCP, and therefore $\tau$-$r$-ACCP.  Hence for $\tau$ refinable, each is $\tau$-$r$-atomic by Theorem \ref{thm: reg-tau} and Theorem \ref{thm: reg}.
\end{corollary}
The following diagram summarizes our results ($\nabla$ represents $\tau$ being refinable):
$$
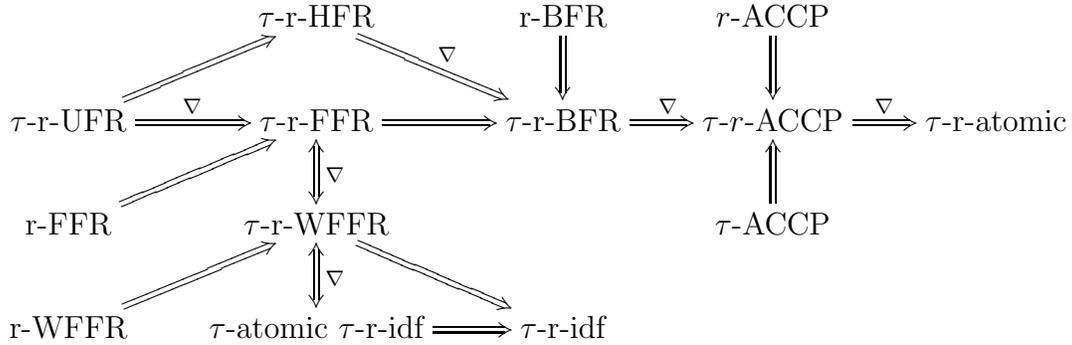
\xymatrix{
            &        \tau\text{-r}\text{-HFR} \ar@{=>}^{\nabla}[dr]     &        \text{r}\text{-BFR} \ar@{=>}[d]       &          r\text{-ACCP}\ar@{=>}[d]        &                 \\
\tau\text{-r}\text{-UFR} \ar@{=>}[ur] \ar@{=>}^{\nabla}[r]& \tau\text{-r}\text{-FFR} \ar@{=>}[r] & \tau\text{-r}\text{-BFR} \ar@{=>}[r]^{\nabla}& \tau\text{-}r\text{-ACCP} \ar@{=>}^{\nabla}[r]& \tau\text{-r}\text{-atomic}\\
 \text{r}\text{-FFR} \ar@{=>}[ur]          &        \tau\text{-r}\text{-WFFR} \ar@{<=>}_{\nabla}[u]  \ar@{=>}[dr]   &             & \tau\text{-ACCP} \ar@{=>}[u]           &                 \\
   \text{r}\text{-WFFR} \ar@{=>}[ur]           &        \tau \text{-atomic  }\tau\text{-r}\text{-idf} \ar@{<=>}_{\nabla}[u]  \ar@{=>}[r]    & \tau\text{-r}\text{-idf}          &            &                }
$$

\begin{lemma} \label{lem: bijection} Let $R$ be a commutative ring with $1$ and let $\tau$ be a symmetric relation on $R^{\#}$.  Let $\alpha \in \{ \emptyset,$ atomic, strongly atomic, m-atomic, unrefinably atomic, very strongly atomic$ \}$.  Every non-unit element in a $\treg$-U-$\alpha$-factorization is an essential divisor.  Moreover, given a $\treg$-$\alpha$-factorization, every $\tau$-factor is essential.  When $\alpha=\emptyset$, we mean simply a $\treg$-U-factorization.  
\end{lemma}
\begin{proof} Let $a\in R$ be a non-unit and let $a=\lambda a_1\cdots a_n \left \lceil b_1 \cdots b_m \right \rceil$ be a $\treg$-U-$\alpha$-factorization.  Then $a=\lambda a_1\cdots a_nb_1\cdots b_m$ is a $\treg$-factorization.  If there is only one $\treg$-factor in the factorization, i.e. $m+n=1$, then this factor is certainly essential.  If it were removed then it would imply that $a$ were a unit, a contradiction.  We now may assume that $m+n \geq 2$, and therefore $a=\lambda a_1\cdots a_nb_1\cdots b_m$ is a $\treg$-factorization implies that $a$ is a product of regular elements and hence is regular.  Moreover, we have $(a)=(b_1\cdots b_m)$ so $ar=b_1\cdots b_m$ for some $r\in R$.  Hence $a=\lambda a_1\cdots a_n \cdot a \cdot r$ and $a$ is regular so cancellation implies that $1=\lambda a_1 \cdots a_n \cdot r$ and in particular $a_i \in U(R)$ for all $1 \leq i \leq n$.  Hence there can be no non-unit inessential $\treg$-divisors as desired.
\\
\indent Given a $\treg$-$\alpha$-factorization of a non-unit $a \in R$, say $a=\lambda a_1 \cdots a_n$, we show that $a_i$ is essential for each $1\leq i \leq n$.  If $n=1$, this is immediate as above.  Thus $n \geq 2$ and therefore $a_i$ is regular for each $1 \geq i\geq n$.  Suppose for a moment that $a_i$ were not essential.  Then $(a)=(a_1 \cdots a_{i-1} \widehat{a_i} a_{i+1} \cdots a_n)=(a_1 \cdots a_n)$.  But this means there is an $r\in R$ such that
$$a_1 \cdots a_{i-1} \widehat{a_i} a_{i+1} \cdots a_n= r\cdot a_1 \cdots a_n.$$
After canceling common factors, since each element on the left is regular, we see that $1=r\cdot a_i$ which means $a_i \in U(R)$, a contradiction since each $a_j \in R^{\#}$ for all $1\leq j \leq n$.  Thus $a_i$ is essential for each $1 \leq i \leq n$ and $\lambda \left \lceil a_1 \cdots a_n \right \rceil$ is indeed a $\tau$-U-$\alpha$ factorization.
\end{proof}
\indent The consequence of this lemma is that we see that $\treg$-$\alpha$-factorizations and $\treg$-U-$\alpha$-factorizations coincide and we see there is a correspondence between the sets given by the map 
$$\phi: \ \{\treg\text{-U-} \alpha \text{-factorizations}\ \}\longrightarrow \ \{\treg\text{-} \alpha \text{-factorizations}\ \}$$
is given by 
$$\lambda a_1 \cdots a_n \left \lceil b_1 \cdots b_m \right \rceil \longmapsto (\lambda a_1 \cdots a_n) b_1 \cdots b_m $$
and the inverse 
$$\phi^{-1}:\ \{\treg\text{-} \alpha \text{-factorizations}\ \}\longrightarrow \ \{\treg\text{-U-} \alpha \text{-factorizations}\ \}$$ 
is given by
$$ \lambda a_1 \cdots a_n \longmapsto \lambda \left \lceil a_1 \cdots a_n \right \rceil.$$
\\
\indent This observation allows us to further consolidate many of our finite factorization properties when it comes to regular factorization.  In particular, we formalize this by way of the following result.
\begin{theorem}  \label{thm: treg-U} Let $R$ be a commutative ring with $1$ and let $\tau$ be a symmetric relation on $R^{\#}$.  Let $\treg:=\tau \cap \left(\reg \times \reg \right)$. Let $\alpha \in \{$ atomic, strongly atomic, m-atomic, unrefinably atomic, very strongly atomic $\}$ and $\beta \in \{$ associate, strongly associate, very strongly associate $\}$.  Then for any choice of $\alpha$ and $\beta$, we have the following.
\\
(1) $R$ is $\treg$-U-$\alpha$ if and only if $R$ is $\treg$-$\alpha$.
\\
(2) $R$ satisfies $\treg$-U-ACCP if and only if $R$ satisfies $\treg$-ACCP.
\\
(3) $R$ is a $\treg$-U-$\alpha$-$\beta$-UFR if and only if $R$ is a $\treg$-$\alpha$-$\beta$-UFR.
\\
(4) $R$ is a $\treg$-U-$\alpha$-HFR if and only if $R$ is a $\treg$-$\alpha$-HFR.
\\
(5) $R$ is a $\treg$-U-BFR if and only if $R$ is a $\treg$-BFR.
\\
(6) $R$ is a $\treg$-U-$\alpha$-$\beta$-df ring if and only if $R$ is a $\treg$-$\alpha$-$\beta$-df ring.
\\
(7) $R$ is a $\treg$-U-$\alpha$, $\treg$-U-$\alpha$-$\beta$-df ring if and only if $R$ is a $\treg$-$\alpha$, $\treg$-$\alpha$-$\beta$-df ring.
\\
(8) $R$ is a $\treg$-U-$\beta$-WFFR if and only if $R$ is a $\treg$-$\beta$-WFFR.
\\
(9) $R$ is a $\treg$-U-$\beta$-FFR if and only if $R$ is a $\treg$-$\beta$-FFR.
\\
\indent If $\tau$ is refinable, then $(6) \Leftrightarrow (7) \Leftrightarrow (8)$.  
\end{theorem}
\begin{proof} (1) ($\Rightarrow$) Let $a \in R$ be a non-unit.  Then there is a $\treg$-U-$\alpha$ factorization of $a$, by Lemma \ref{lem: bijection}, this factorization is of the form $a=\lambda\left \lceil a_1 \cdots a_n \right \rceil$.  By definition, $a=\lambda a_1 \cdots a_n$ is a $\treg$-factorization and $a_i$ is $\treg$-$\alpha$ for each $1\leq i \leq n$and therefore this is a $\treg$-$\alpha$-factorization of $a$.  ($\Leftarrow$) This is shown in \cite[Theorem 4.3]{Mooney2}.
\\
\indent (2) ($\Rightarrow$) Let $a\in R$ be a non-unit.  Suppose there was an ascending chain of principal ideals of the form $(a) \subsetneq (a_1) \subsetneq (a_2) \subsetneq \cdots $ such that $a_{i+1} \mid_{\treg} a_i$ for each $i$.  Say the $\treg$-factorization for each $i$ is given by
$$a_{i}=\lambda a_{i+1}a_{i1}\cdots a_{in_i}.$$
because $a_i \subsetneq a_{i+1}$, we know that this $\treg$-factorization is non-trivial and therefore each $\treg$-factor is regular, in particular $a_i$ is regular, and therefore by Lemma \ref{lem: bijection}, is essential.  This would contradict the fact that $R$ satisfies $\treg$-U-ACCP. ($\Leftarrow$) This is shown in \cite[Theorem 4.3]{Mooney2}.
\\
\indent (3) (resp. (4)) Let $a\in R$ be a non-unit.  Then by Lemma \ref{lem: bijection}, $a$ has a $\treg$-U-$\alpha$ factorization if and only if $a$ has a $\treg$-$\alpha$-factorization.  Furthermore, since the $\treg$-U-factorizations have no inessential divisors, it is clear that the equivalence of the uniqueness (resp. constant length) of these factorizations follows as well.
\\
\indent (5) and (9) Let $a\in R$ be a non-unit.  By Lemma \ref{lem: bijection}, the correspondence shows that we may apply $\phi^{-1}$ to any $\treg$-factorization of $a$ of length $n$ and get a $\treg$-U-factorization with the same $n$ $\treg$-factors all occuring as the $\treg$ essential divisors in the corresponding $\treg$-U-factorization.  Similarly, given a $\treg$-U-factorization with $n$ essential divisors, we may apply $\phi$ to this factorization and get a $\treg$-factorization of length $n$ with the same $\treg$-factors as the essential $\treg$-divisors.  Hence there is a bound on the length of the number of essential divisors in any $\treg$-U-factorization of $a$ if and only if there is a bound on the length of any $\treg$-factorization of $a$. Moreover, this same correspondence shows that there are the same number of $\treg$-factorizations of $a$ up to $\beta$ as there are $\treg$-U-factorizations of $a$ up to $\beta$.  
\\
\indent (6) (resp. (8) Let $a\in R$ be a non-unit.  Let $a\in R$ be a non-unit.  As in the proof of (5) and (9), it is clear that the set of $\treg$-divisors and essential $\treg$-divisors of $a$ are the same by the correspondence given in Lemma \ref{lem: bijection} and map $\phi$.  This means the set of $\treg$-divisors of $a$ and essential $\treg$-divisors of $a$ up to $\beta$ are the same.  Moreover, this also means that the set of $\treg$-$\alpha$ divisors and the set of $\treg$-$\alpha$-essential divisors are the same up to $\beta$ as well.
\\
\indent (7) This follows immediately by combining the results of part (1) and (6). 
\end{proof}
\indent We can further relate the various properties by removing the very strongly atomic choice for $\alpha$ and the very strongly associate choice for $\beta$ in the above theorem.  This will allow us to combine the result of Theorem \ref{thm: treg-U} into a single theorem below.
\begin{corollary}  Let $R$ be a commutative ring with $1$ and let $\tau$ be a symmetric relation on $R^{\#}$.  Let $\treg:=\tau \cap \left(\reg \times \reg \right)$. Let $\alpha \in \{$ atomic, strongly atomic, m-atomic, unrefinably atomic $\}$ and $\beta \in \{$ associate, strongly associate $\}$.  Then for any choice of $\alpha$ and $\beta$, we have the following.
\\
(1) $R$ is $\treg$-U-$\alpha$ if and only if $R$ is $\treg$-$\alpha$ if and only if $R$ is $\tau$-regular-atomic.
\\
(2) $R$ satisfies $\treg$-U-ACCP if and only if $R$ satisfies $\treg$-ACCP if and only if $R$ satisfies $\tau$-regular-ACCP.
\\
(3) $R$ is a $\treg$-U-$\alpha$-$\beta$-UFR if and only if $R$ is a $\treg$-$\alpha$-$\beta$-UFR if and only if $R$ is $\tau$-regular-UFR.
\\
(4) $R$ is a $\treg$-U-$\alpha$-HFR if and only if $R$ is a $\treg$-$\alpha$-HFR if and only if $R$ is $\tau$-regular-HFR.
\\
(5) $R$ is a $\treg$-U-BFR if and only if $R$ is a $\treg$-BFR if and only if $R$ is $\tau$-regular-BFR.
\\
(6) $R$ is a $\treg$-U-$\alpha$-$\beta$-df ring if and only if $R$ is a $\treg$-$\alpha$-$\beta$-df ring if and only if $R$ is a $\tau$-regular-idf ring.
\\
(7) $R$ is a $\treg$-U-$\alpha$, $\treg$-U-$\alpha$-$\beta$-df ring if and only if $R$ is a $\treg$-$\alpha$, $\treg$-$\alpha$-$\beta$-df ring if and only if $R$ is $\tau$-regular-atomic, $\tau$-regular-idf ring.
\\
(8) $R$ is a $\treg$-U-$\beta$-WFFR if and only if $R$ is a $\treg$-$\beta$-WFFR if and only if $R$ is $\tau$-regular-WFFR.
\\
(9) $R$ is a $\treg$-U-$\beta$-FFR if and only if $R$ is a $\treg$-$\beta$-FFR if and only if $R$ is $\tau$-regular-FFR.
\\
\indent If $\tau$ is refinable, then $(6) \Leftrightarrow (7) \Leftrightarrow (8)$.  
\end{corollary}
\begin{proof} The first equivalence in each statement (i) for $1 \leq i \leq 9$ follows directly from Theorem \ref{thm: treg-U}.  Similarly, the second equivalence in each statement (i) for $1 \leq i \leq 9$ follows from Theorem \ref{thm: treg-UFR}.  
\end{proof}
\indent We conclude the article with a diagram which summarizes many of the equivalences and relationships demonstrated thus far where $\treg$ is defined as above, $\alpha \in \{$atomic, strongly atomic, m-atomic, unrefinably atomic $\}$, $\beta \in \{$ associate, strongly associate $\}$, and $\nabla$ represents $\tau$ is refinable.
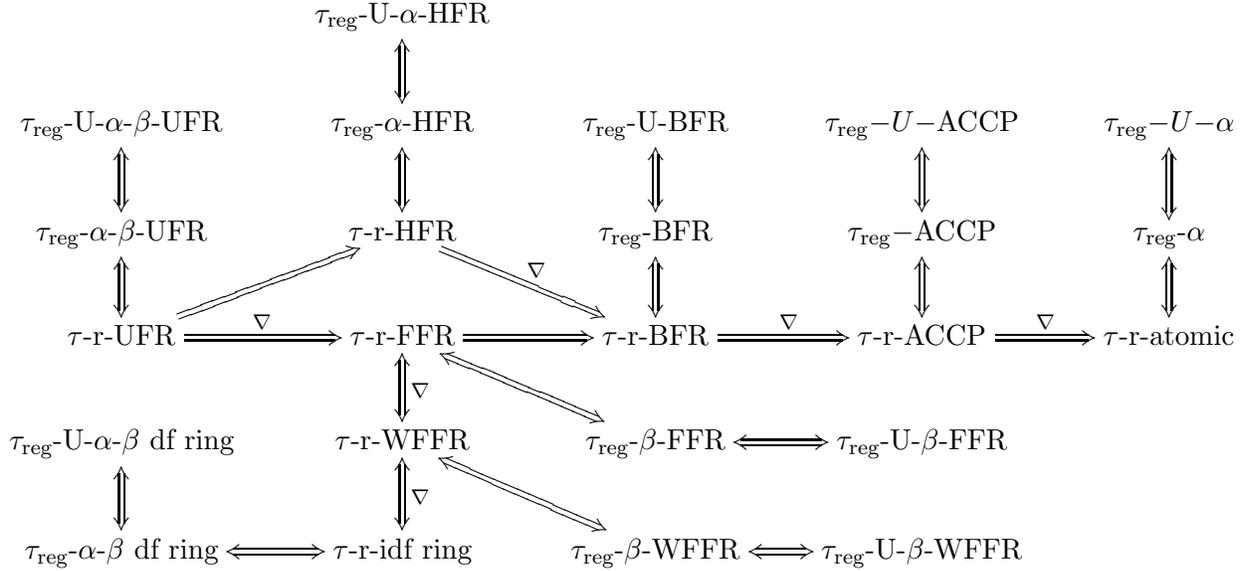
\small{$$\xymatrix{   &   \treg \text{-U-} \alpha \text{-HFR} \ar@{<=>}[d] &   &  & \\
 \treg \text{-U-} \alpha \text{-} \beta \text{-UFR}\ar@{<=>}[d]  &  \treg \text{-} \alpha \text{-HFR} \ar@{<=>}[d]  & \treg \text{-U-} \text{BFR} \ar@{<=>}[d]   &  \treg{-U-} \text{ACCP} \ar@{<=>}[d] &   \treg{-U-} \alpha \ar@{<=>}[d] \\
     \treg \text{-} \alpha \text{-} \beta \text{-UFR} \ar@{<=>}[d]       &        \tau \text{-r-HFR} \ar@{=>}^{\nabla}[dr]     &    \treg \text{-} \text{BFR} \ar@{<=>}[d]       &   \treg{-} \text{ACCP} \ar@{<=>}[d]   &         \treg \text{-} \alpha \ar@{<=>}[d]                      \\
\tau \text{-r-UFR} \ar@{=>}[ur] \ar@{=>}^{\nabla}[r]  & \tau \text{-r-FFR} \ar@{=>}[r] \ar@{<=>}^{\nabla}[d]  & \tau \text{-r-BFR} \ar@{=>}[r]^{\nabla}& \tau\text{-r-ACCP} \ar@{=>}^{\nabla}[r] & \tau\text{-r-atomic} \\
  \treg\text{-U-}\alpha \text{-} \beta \text{ df ring}   \ar@{<=>}[d]         & \tau \text{-r-WFFR} \ar@{<=>}^{\nabla}[d]                     &   \treg\text{-}\beta \text{-FFR} \ar@{<=>}[ul]     &  \treg\text{-U-}\beta \text{-FFR} \ar@{<=>}[l]                     \\
 \treg\text{-}\alpha \text{-} \beta \text{ df ring}   \ar@{<=>}[r]        & \tau\text{-}\text{r-idf ring} & \treg \text{-} \beta \text{-WFFR} \ar@{<=>}[ul]    &\treg \text{-U-} \beta \text{-WFFR} \ar@{<=>}[l] &   
            }$$}

\section*{Acknowledgments} 
The author would like to acknowledge both The University of Iowa and Viterbo University for their support while the research involved in this particular article was completed during employment at both universities.

\bibliographystyle{plain}
\bibliography{bibliography}

\end{document}